\documentclass[12pt,a4paper,reqno]{amsart}
\usepackage{lineno}
\allowdisplaybreaks
\usepackage{amsmath}
\usepackage{amssymb}
\usepackage{amsfonts}
\usepackage{amsthm,amsfonts,amsthm,latexsym,enumerate,url,cases}
\numberwithin{equation}{section}
\usepackage{mathrsfs}
\usepackage{hyperref}
\hypersetup{colorlinks=true,citecolor=blue,linkcolor=blue,urlcolor=blue}
\def\pmod #1{\ ({\rm{mod}}\ #1)}

\theoremstyle{plain}
\newtheorem{theorem}{Theorem}

\newtheorem{lemma}{Lemma}
\newtheorem{problem}{Problem}

\newtheorem{conjecture}{Conjecture}
\theoremstyle{definition}

\newtheorem{remark}{Remark}

\usepackage{etoolbox}
\makeatletter
\patchcmd{\@settitle}{\uppercasenonmath\@title}{}{}{}
\patchcmd{\@setauthors}{\MakeUppercase}{}{}{}
\patchcmd{\section}{\scshape}{}{}{}
\makeatother

\begin{document}

\title
[On a conjecture on Romanoff type sumsets]
{On a conjecture on Romanoff type sumsets}

\author
[Y. Ding and L. Li] {Yuchen Ding and Liangxun Li}
\address{(Yuchen Ding$^{1,2}$) $^1$School of Mathematics,  Yangzhou University, Yangzhou 225002, People's Republic of China}
\address{$^2$HUN-REN Alfr\'ed R\'enyi Institute of Mathematics, Budapest, Pf. 127, H-1364 Hungary}
\email{ycding@yzu.edu.cn}

\address{(Liangxun Li$^{1,2}$) $^1$Mathematical Research Center, Shandong University, Jinan, Shandong 250100, People's Republic of China}
\address{$^2$HUN-REN Alfr\'ed R\'enyi Institute of Mathematics, Budapest, Pf. 127, H-1364 Hungary}
\email{lxli@mail.sdu.edu.cn}



\begin{abstract}
In this note, we generalize a 1950 result of P. Erd\H os on upper bounds of $k$-th moment of Romanoff type representation functions. As an application, we give a conditional proof of a recent conjecture of Y.-G. Chen on Romanoff type sumsets under the assumption of the Hardy-Littlewood conjecture.
\end{abstract}
\maketitle

\section{Introduction}
Let 
$
\mathcal{R}:=\{r_1<r_2<\cdots\}
$
be the set of odd integers which can be written as the sum of a prime and a power of two. 
Let $\mathbb{N}_{\text{odd}}$ be the set of odd numbers and
$$
\mathcal{N}=\mathbb{N}_{\text{odd}}\setminus \mathcal{R}.
$$
The old letters between Euler and Goldbach \cite{Euler} as well as de Polignac's conjecture \cite{Polignac1,Polignac2} inspired Romanoff \cite{Romanoff} to prove that there is an absolute constant $c_1>0$ such that 
$$
\mathcal{R}(x)>c_1x
$$
for all sufficiently large $x$, where $\mathcal{R}(x)$ is the number of elements of $\mathcal{R}$ not exceeding $x$.
Later, van der Corput \cite{Corput} proved that there is an absolute constant $c_2>0$ such that
$$
\mathcal{N}(x)>c_2x
$$
for all sufficiently large $x$.
Introducing concepts of the covering congruence systems Erd\H os \cite{Erdos2} showed that $\mathcal{N}$ contains an arithmetic progression. Erd\H os \cite{Erdos-95} asked whether the set $\mathcal{N}$ is composed by the union of an arithmetic progression and a zero density set. Very recently, Chen's remarkably interesting result \cite{Chen2} shows that the set $\mathcal{N}$ is not composed by the union of a finite number of arithmetic progressions and a zero density set, thus answering negatively Erd\H os' problem in a rather satisfactory way.

Erd\H os' result above implies that the gaps between the elements of $\mathcal{N}$ are bounded. In \cite{Chen}, Chen proved that
$$
\limsup_{n\rightarrow\infty}(r_{n+1}-r_n)\ge 8.
$$
Answering in the affirmative a 1977 problem of Erd\H os \cite[Page 51]{Erdos-77}, very recently Elsholtz, Planitzer and Schlage-Puchta \cite{EPSP} (via private communications) proved that
$$
\limsup_{n\rightarrow\infty}(r_{n+1}-r_n)=\infty.
$$
Actually, they gave a quantitative form involving the maximal order of $r_{n+1}-r_n$. The result of Elsholtz, Planitzer and Schlage-Puchta is rather surprising since the predicted lower density of $\mathcal{R}$ is pretty high approximately $0.437...$, thanks to the article of Del Corso, Corso, Dvornicich and Romani \cite{CCDR}. At present, the best record of the lower density of $\mathcal{R}$ is $0.110114$ given by Elsholtz and Schlage-Puchta \cite{ESP}.  

Let $\mathbb{N}$ be the set of natural numbers and $\mathbb{P}$ the set of primes.
Improving the lower bound of Yang and Chen \cite{Yang-Chen}, the first named author \cite{Ding} proved 
$$
\#\left\{n\le x: n=p+2^{m_1^2}+2^{m_2^2}, m_1, m_2\in \mathbb{N}, p\in \mathbb{P}\right\}\gg x.
$$
Let $r_1,r_2,\cdots, r_t$ be $t$ given positive integers. Chen and Xu \cite{Chen-Xu} proved that
$$
\#\left\{n\le x: n=p+2^{ m_1^{r_1}}+\cdots+2^{ m_t^{r_t}}, m_i\in \mathbb{N}, 1\le i\le r, p\in \mathbb{P}\right\}\gg x,
$$
provided that $r_1^{-1}+\cdots+r_t^{-1}\ge 1$. Their result was extended by Zhai and the first named author \cite{Ding-Zhai} to real cases which means that all of $r_i$ are allowed to be real numbers. In the extended real case, all the summands $2^{m_i^{r_i}}$ in Chen-Xu's result should be naturally adjusted to $2^{\lfloor m_i^{r_i}\rfloor}$. A shorter interval version of this representation was considered by Verwee and the first named author very recently \cite{Ding-V}. The representations of odd integers of the form $p+2^{m_1^2}+2^{m_2^2}$ were further investigated by Xu and Chen in two subsequent articles \cite{Xu-Chen1,Xu-Chen2}.

Recently, in \cite{Chen-new} Chen obtained some new results involving odd numbers of the representations 
$$
p+2^{ m_1^{r_1}}+\cdots+2^{ m_t^{r_t}}.
$$
He proved that almost all odd integers of this form have unique representations, provided that $r_1^{-1}+\cdots+r_t^{-1}< 1$. This phenomenon is not isolated. Actually, previously Chen \cite{Chen-old} proved that almost all odd numbers of the form $p+2^q$ have unique representations, where $p$ and $q$ are both primes.
From now on, let $r_1,\cdots, r_t$ be positive numbers with $r_i\ge 1~(1\le i\le t)$ and $r_1^{-1}+\cdots+r_t^{-1}\ge 1$. Set
\begin{align}\label{eq-1-1}
\mathcal{A}=\left\{2^{ \left\lfloor m_1^{r_1}\right\rfloor}+\cdots+2^{\left\lfloor m_t^{r_t}\right\rfloor}: m_i\in \mathbb{N}, 1\le i\le t\right\}
\end{align}
and
\begin{align}\label{eq-1-2}
\mathcal{R}_{\mathcal{A}}=\mathbb{P}+\mathcal{A}:=\big\{p+a: p\in \mathbb{P}, a\in \mathcal{A}\big\}.
\end{align}
Chen \cite[Conjecture 1.7]{Chen-new} made the following conjecture.

\begin{conjecture}[Chen]\label{conjecture-1}
If $r_1,\cdots, r_t$ are positive integers with $r_1^{-1}+\cdots+r_t^{-1}\ge 1$, then
 the set $\mathcal{R}_{\mathcal{A}}\cap(\mathcal{R}_{\mathcal{A}}+2)$ has a positive lower asymptotic density, where
 $$
 \mathcal{R}_{\mathcal{A}}+2=\big\{n+2: n\in\mathcal{R}_{\mathcal{A}}\big\}.
 $$
\end{conjecture}

The purpose of this note is to provide a conditional proof of Conjecture \ref{conjecture-1} under the assumption of a uniform Hardy-Littlewood conjecture (see e.g., \cite{Soundararajan}).

\begin{conjecture}[Weak uniform Hardy-Littlewood Conjecture]\label{conjecture-2}
There is a fixed $\delta>0$ such that the inequality
$$
\#\big\{p\le x: p, p+h\in \mathbb{P}\big\}\gg_\delta \frac{x}{(\log x)^2}
$$
holds uniformly for any even $h\le x^\delta$, provided that $x$ is sufficiently large.
\end{conjecture}

\begin{remark}\label{remark1}
The Hardy-Littlewood Conjecture on generalized twin primes was initially posed by Hardy and Littlewood \cite{Hardy}, which states for even $h$
$$
\#\big\{p\le x: p, p+h\in \mathbb{P}\big\}\sim \mathfrak{S}_h\frac{x}{(\log x)^2}
$$
as $x\to\infty$, where
$$
\mathfrak{S}_h:=\prod_{p\geq 3\atop p\nmid h}\Big(1-\frac{1}{(p-1)^2}\Big)\prod_{p\geq 2\atop p\mid h}\Big(1-\frac{1}{p}\Big)^{-1},
$$
which obeys the uniform lower bound $\mathfrak{S}_h\geq \mathfrak{S}_2\gg 1.$
 For the proof of our theorem, instead of the original asymptotic formula of generalized twin primes we need this weak uniform version of the Hardy-Littlewood Conjecture which only admits a lower bound. The uniformity of $h\le x^\delta$ for some small $\delta>0$ is not too far-fetched. In fact, Montgomery and Soundararajan \cite{Montgomery} investigated Cram\'er's conjecture \cite{Cramer} on prime gaps assuming a variant of the Hardy-Littlewood conjecture under the uniformity of $h\le x^{1-\varepsilon}$ for any $\varepsilon>0$. An even stronger conjecture was assumed in an article of Lemke Olivera and Soundararajan \cite{Soundararajan}. Maynard's remarkable results on prime tuples \cite{Maynard1,Maynard2} appear not to be applicable directly in this context (but we are not sure whether the method leading to the results therein could be used here). Actually, we do not even know how to prove unconditionally that the set $\mathcal{R}_{\mathcal{A}}\cap(\mathcal{R}_{\mathcal{A}}+2)$ contains infinitely many odd numbers. For the particular case $\mathcal{A}=\big\{2^m:m\in \mathbb{N}\big\}$, it is clear that there are $\gg x/\log x$ odd numbers $n=p+2$ such that $n+2=p+2^2$ are also the sums of a prime and a power of two.
\end{remark}

We now state our main result as the following theorem. 

\begin{theorem}\label{theorem-1}
Let $r_1,\cdots, r_t$ are positive numbers with $r_i\ge 1~(1\le i\le t)$ and 
$$
r_1^{-1}+\cdots+r_t^{-1}\ge 1.
$$
 Then
 the set $\mathcal{R}_{\mathcal{A}}\cap(\mathcal{R}_{\mathcal{A}}+2)$ has a positive lower asymptotic density under the assumption of Conjecture \ref{conjecture-2}.
\end{theorem}

It is worth emphasizing that we do not require $r$'s in Theorem \ref{theorem-1} to be integers which covers a wider range than the original conjecture.

\section{Proof of Theorem \ref{theorem-1}}
The proof of Theorem \ref{theorem-1}, among other things, is based on an extended high moments of representation functions which generalizes a 1950 result of Erd\H os \cite{Erdos2}.

For odd integer $n$, let
$$
f(n)=\#\left\{(p,m)\in \mathbb{P}\times \mathbb{N}:n=p+2^m\right\}.
$$
By the Cauchy-Schwarz inequality and the prime number theorem, Romanoff's theorem is equivalent to
$
\sum_{n\le x}f(n)^2\ll x.
$
 Erd\H os \cite{Erdos2} then took a further step, and established
$$
\sum_{n\le x}f(n)^k\ll_k x
$$
for any fixed integer $k$. Now, for odd integer $n$ let
$$
f_{\mathcal{A}}(n)=\#\left\{(p,a)\in \mathbb{P}\times \mathcal{A}:n=p+a\right\},
$$
where the set $\mathcal{A}$ is defined by (\ref{eq-1-1}). Motivated by Erd\H os, we will prove the following.

\begin{theorem}\label{propo:2}
For any fixed positive integer $k$, we have
$$
\sum_{n\le x}f_{\mathcal{A}}(n)^k\ll x (\log x)^{k\left(\frac{1}{r_1}+\cdots+\frac{1}{r_t}-1\right)},
$$
where the implied constant depends at most on $\mathcal{A}$ and $k$.
\end{theorem}

The arguments in \cite{Chen-Xu, Ding-Zhai} could essentially lead to the proof of Theorem \ref{propo:2} with $k=2$. But, for $k=2$ these two articles are slightly different both in their forms of presentations and in their proofs. We believe that the generality and concise form of Theorem \ref{propo:2} have the potential for further applications in Romanoff type sumsets.

We postpone the proof of Theorem \ref{propo:2} to the last section, and give first the proof of Theorem \ref{theorem-1} by Conjecture \ref{conjecture-2} and Theorem \ref{propo:2}. 
Let 
$$
\mathcal{W}_r=\{j: r_j=r, 1\le j\le t\},\ w_r=|\mathcal{W}_r|,\ \text{and}~w=\prod_{r=1}^{\infty}w_r!.
$$
Chen \cite[Lemma 2.1]{Chen-new} proved that
$$
\mathcal{A}(x)\sim \frac{1}{w}\left(\frac{\log x}{\log 2}\right)^{\frac{1}{r_1}+\cdots+\frac{1}{r_t}},
$$
provided that $r_1,\cdots, r_t$ are positive integers. For our purpose, a lower bound of $\mathcal{A}(x)$ which allows $r$'s to be positive numbers will be applicable enough.

\begin{lemma}\label{lem1}
Let $r_1,\cdots, r_t$ be positive number. Then for sufficiently large $x$ we have
$$
\mathcal{A}(x)\gg (\log x)^{\frac{1}{r_1}+\cdots+\frac{1}{r_t}},
$$
where the implied constant depends only on $r_1,\cdots, r_t$.
\end{lemma}
\begin{proof}
For any $1\le i\le t$, we consider the number of integers $m_i$ such that
\begin{align}\label{eq-lem-1}
x^{\frac{i-1}{t+1}}\le 2^{\left\lfloor m_i^{r_i}\right\rfloor}< x^{\frac{i}{t+1}}.
\end{align}
It is clear that the number of such $m_i$ is 
$
\gg (\log x)^{\frac{1}{r_i}}.
$
Moreover, for integers $m_i$ in the range (\ref{eq-lem-1}) it is easy to see that all the sums below
$$
2^{\left\lfloor m_1^{r_1}\right\rfloor}+\cdots+2^{\left\lfloor m_t^{r_t}\right\rfloor}
$$
are different because of the unique binary representation of integers. Therefore,
$$
\mathcal{A}(x)\gg \prod_{1\le i\le t}(\log x)^{\frac{1}{r_i}}=(\log x)^{\frac{1}{r_1}+\cdots+\frac{1}{r_t}},
$$
proving our lemma.
\end{proof}

\begin{proof}[Proof of Theorem \ref{theorem-1}]
Let $\mathcal{R}_{\mathcal{A},2}=\mathcal{R}_{\mathcal{A}}\cap(\mathcal{R}_{\mathcal{A}}+2)$ for simplification. Using twice the Cauchy-Schwarz inequality we have
\begin{align}\label{eq-2-1}
\Big(\sum_{n\le x} f_{\mathcal{A}}(n)f_{\mathcal{A}}(n+2)\Big)^2&\le \Big(\sum_{\substack{n\le x\\ f_{\mathcal{A}}(n)f_{\mathcal{A}}(n+2)\ge 1}} 1^2 \Big)\Big(\sum_{n\le x} f_{\mathcal{A}}(n)^2f_{\mathcal{A}}(n+2)^2\Big)\nonumber\\
&\le \mathcal{R}_{\mathcal{A},2}(x)\sqrt{\sum_{n\le x} f_{\mathcal{A}}(n)^4}\sqrt{\sum_{n\le x} f_{\mathcal{A}}(n+2)^4}.
\end{align}
By Theorem \ref{propo:2}, we know
\begin{align}\label{eq-2-2}
\sum_{n\le x} f_{\mathcal{A}}(n)^4\ll x (\log x)^{4\left(\frac{1}{r_1}+\cdots+\frac{1}{r_t}-1\right)}
\end{align}
and
\begin{align}\label{eq-2-3}
\sum_{n\le x} f_{\mathcal{A}}(n+2)^4\le \sum_{n\le 2x} f_{\mathcal{A}}(n)^4 \ll x (\log x)^{4\left(\frac{1}{r_1}+\cdots+\frac{1}{r_t}-1\right)}.
\end{align}
Inserting  (\ref{eq-2-2}) and (\ref{eq-2-3}) into (\ref{eq-2-1}), we get
$$
\mathcal{R}_{\mathcal{A},2}(x)\gg \frac{\Big(\sum_{n\le x} f_{\mathcal{A}}(n)f_{\mathcal{A}}(n+2)\Big)^2}{x(\log x)^{4\left(\frac{1}{r_1}+\cdots+\frac{1}{r_t}-1\right)}}.
$$
Thus, to prove our theorem it suffices to show that
\begin{align}\label{eq-2-4}
\sum_{n\le x} f_{\mathcal{A}}(n)f_{\mathcal{A}}(n+2)\gg x(\log x)^{2\left(\frac{1}{r_1}+\cdots+\frac{1}{r_t}-1\right)}.
\end{align}

From now on, the symbols $p$'s will always be primes. Clearly,
\begin{align*}
\sum_{n\le x} f_{\mathcal{A}}(n)f_{\mathcal{A}}(n+2)&=\sum_{n\le x}\sum_{\substack{n=p_1+a_1\\ n+2=p_2+a_2\\ a_1, a_2\in \mathcal{A}}}1\\
&=\sum_{a_1, a_2\in \mathcal{A}}\sum_{p_1+a_1=p_2+a_2-2\le x}1\\
&\ge \sum_{\substack{a_1, a_2\in \mathcal{A}\\a_2<a_1\le x^\delta}}\sum_{\substack{p_1\le \frac{x}{2}\\ p_2=p_1+a_1-a_2+2}}1
\end{align*}
where $\delta$ is the one occurred in Conjecture \ref{conjecture-2}. By Conjecture \ref{conjecture-2}, we have
\begin{align*}
\sum_{\substack{p_1\le \frac{x}{2}\\ p_2=p_1+a_1-a_2+2}}1\gg \frac{x}{(\log x)^2}.
\end{align*}
Therefore, we conclude that
\begin{align*}
\sum_{n\le x} f_{\mathcal{A}}(n)f_{\mathcal{A}}(n+2)\gg \frac{x}{(\log x)^2}\sum_{\substack{a_1, a_2\in \mathcal{A}\\a_2<a_1\le x^\delta}}1=\frac{x}{(\log x)^2}\binom{\mathcal{A}\left(x^\delta\right)}{2}.
\end{align*}
By Lemma \ref{lem1} we know that
$$
\mathcal{A}\left(x^\delta\right)\gg (\log x)^{\frac{1}{r_1}+\cdots+\frac{1}{r_t}},
$$
from which we deduce that
$$
\sum_{n\le x} f_{\mathcal{A}}(n)f_{\mathcal{A}}(n+2)\gg x(\log x)^{2\left(\frac{1}{r_1}+\cdots+\frac{1}{r_t}-1\right)},
$$
proving (\ref{eq-2-4}).
\end{proof}

\section{$k$-th moment of Romanoff type representation functions}
In this section, we offer the proof of Theorem \ref{propo:2}.
We need a few more lemmas. The first one is a standard sieve result, see e.g., \cite[Theorem 2.4, p. 76]{Halberstam}.

\begin{lemma}\label{lem-2}
Let $g$ be a natural number, and let $q_i, b_i~(i=1, 2, \cdots, g)$ be integers satisfying
$$
E:=\prod_{1\le i\le g}q_i\prod_{1\le h<s\le g}\left(q_hb_s-q_sb_h\right)\neq 0.
$$
Let $\rho(p)$ be the number of solutions in $n$ modulo $p$ of
$$
\prod_{1\le i\le g}\big(q_in+b_i\big)\equiv 0\pmod{p}.
$$
Then
$$
\#\big\{p\le y: q_ip+b_i\in \mathbb{P}, 1\le i\le g\big\}\ll \prod_{p|E}\left(1-\frac{1}{p}\right)^{\rho(p)-g}\prod_{p|b_1\cdots b_g}\left(1-\frac{1}{p}\right)^{-1}\frac{y}{(\log y)^{g+1}},
$$
where the implied constant depends at most on $g$.
\end{lemma}

Extending an old result of Erd\H os and Tur\'an \cite{Erdos-Turan} as well a result of Erd\H os \cite[Eq. (14)]{Erdos2}, we give the following lemma.

\begin{lemma}\label{lem-ET}
For any odd integer $d$, let $e_2(d)$ be the least positive integer $m$ such that $2^m\equiv 1\pmod{d}$. Let $B$ and $\varepsilon$ be two given positive numbers. Then
$$
\sum_{d=1,\ 2\nmid d}^{\infty}\frac{\mu(d)^2B^{\omega(d)}}{d\big(e_2(d)\big)^\varepsilon}<\infty.
$$
\end{lemma}
\begin{proof}
For $k=e_2(d)$, by the definition we have $2^k\equiv 1\pmod{d}$. Hence,
\begin{align}\label{eq-ET-0}
\sum_{d=1,\ 2\nmid d}^{\infty}\frac{\mu(d)^2B^{\omega(d)}}{d\big(e_2(d)\big)^\varepsilon}=\sum_{k=2}^{\infty}\frac{1}{k^\varepsilon}\sum_{\substack{d=1\\ e_2(d)=k}}^{\infty}\frac{\mu(d)^2B^{\omega(d)}}{d}.
\end{align}
It is well-known (see, e.g., \cite{Montgomery2}) that there is an absolute constant $c_0>e^e$ such that
$$
\omega(d)<c_0\log d,
$$
provided that $d>1$. For any $t\ge 2$ we have
\begin{align}\label{eq-ET-1}
g(t):=\sum_{k\le t}\sum_{\substack{d=1\\ e_2(d)=k}}^{\infty}\frac{\mu(d)^2B^{\omega(d)}}{d}&\le \sum_{\substack{d=1\\ d| \prod_{k\le t}\left(2^k-1\right)}}^{\infty}\frac{\mu(d)^2B^{\omega(d)}}{d}\nonumber\\
&\le \sum_{\substack{d\le 2^{t^2}}}^{\infty}\frac{\mu(d)^2B^{\omega(d)}}{d}\nonumber\\
&\le \prod_{p\le c_0t^2\log 2}\left(1+\frac{B}{p}\right)\nonumber\\
&=\exp\bigg(\sum_{p\le c_0t^2\log 2}\log \Big(1+\frac{B}{p}\Big)\bigg)\nonumber\\
&<\exp\Big(B\sum_{p\le c_0t^2\log 2} 1/p\Big).
\end{align}
By Mertens formula (see, e.g.,  \cite[Theorem 2.7]{Montgomery2}),  we have
\begin{align}\label{eq-ET-2}
\sum_{p\le c_0t^2\log 2}\frac{1}{p}=\log\log (c_0t^2\log 2)+\gamma+O(1/\log k)< \widetilde{c_0}\log\log t,
\end{align}
where $\gamma$ is the Meissel–Mertens constant and $\widetilde{c_0}$ is an absolute constant.
Combining (\ref{eq-ET-1}) with (\ref{eq-ET-2}), for any $t\ge 2$ we get
\begin{align}\label{eq-ET-3}
\sum_{k\le t}\sum_{\substack{d=1,\ d|2^k-1}}^{\infty}\frac{\mu(d)^2B^{\omega(d)}}{d}\le (\log t)^{B_1},
\end{align}
where $B_1=B\widetilde{c_0}$. Hence, by (\ref{eq-ET-0}), (\ref{eq-ET-3}) and partial summations we have
\begin{align*}
\sum_{d=1,\ 2\nmid d}^{\infty}\frac{\mu(d)^2B^{\omega(d)}}{d\big(e_2(d)\big)^\varepsilon}=\varepsilon\int_{2}^{\infty}
\frac{g(t)}{t^{1+\varepsilon}}{\rm d}t<\infty,
\end{align*}
proving our lemma.
\end{proof}

The following technical lemma is quoted from \cite[inequality between (2.16) and (2.17), p. 171]{Ding-Zhai}.
\begin{lemma}\label{lem-3}
Let $(\kappa, \lambda)$ be an exponent pair and $d$ an odd integer.  Suppose that $r_1>1$ is not an integer. Then 
$$
\max_{f}\sum_{\substack{2^{ \lfloor m_1^{r_1}\rfloor}\equiv f\pmod{d}\\ m_1\le (\log 2x/\log 2)^{\frac{1}{r_1}}}}1\ll \frac{(\log x)^{\frac{1}{r_1}}}{e_2(d)}+(\log x)^{\frac{r_1\kappa+\lambda}{r_1(1+\kappa)}}e_2(d)^{-\frac{\kappa}{1+\kappa}}+e_2(d)^{\frac{1}{r_1}},
$$
where the implied constant depends at most on $r_1$.
\end{lemma}

Let $\mu(d)$ be the M\"obius function, $\omega(d)$ the number of different prime factors of $d$, and $P^+(d)$ the largest prime factor of $d$.

\begin{lemma}\label{lem-4}
Let $k>0 $ be a given number. Suppose that $r_1>1$ is not an integer. Then we have
$$
\sum_{\substack{d=1,\ 2\nmid d \\P^+(d)\le \log x}}^{\infty}\frac{\mu(d)^2 k^{3\omega(d)}}{d}\max_{f}\sum_{\substack{2^{ \lfloor m_1^{r_1}\rfloor}\equiv f\pmod{d}\\ m_1\le (\log 2x/\log 2)^{\frac{1}{r_1}}}}1\ll 
(\log x)^{\frac{1}{r_1}},
$$
provided that $x$ is sufficiently large, where the implied constant depends at most on $r_1$.
\end{lemma}
\begin{proof}
Let $\mathcal{D}$ be the set of $d$ such that $2\nmid d$ and 
$$
\max_{f}\sum_{\substack{2^{ \lfloor m_1^{r_1}\rfloor}\equiv f\pmod{d}\\ m_1\le (\log 2x/\log 2)^{\frac{1}{r_1}}}}1 \le \frac{(\log x)^{\frac{1}{r_1}}}{(\log\log x)^{k^3}}.
$$
By the Mertens estimate, we have
\begin{align*}
\prod_{p\le \log x}\left(1+\frac{k^3}{p}\right)&=\exp\left(\sum_{p\le \log x}\log (1+k^3/p)\right)<\exp\left(k^3\sum_{p\le \log x}1/p\right)\ll (\log\log x)^{k^3},
\end{align*}
from which we deduce that
\begin{align*}
\sum_{\substack{d=1,\ d\in \mathcal{D} \\P^+(d)\le \log x}}^{\infty}\frac{\mu(d)^2 k^{3\omega(d)}}{d}\max_{f}\sum_{\substack{2^{ \lfloor m_1^{r_1}\rfloor}\equiv f\pmod{d}\\ m_1\le (\log 2x/\log 2)^{\frac{1}{r_1}}}}1  &\le \frac{(\log x)^{\frac{1}{r_1}}}{(\log\log x)^{k^3}}\sum_{\substack{d=1,\ d\in \mathcal{D} \\P^+(d)\le \log x}}^{\infty}\frac{\mu(d)^2 k^{3\omega(d)}}{d}\\
&\le \frac{(\log x)^{\frac{1}{r_1}}}{(\log\log x)^{k^3}}\prod_{p\le \log x}\left(1+\frac{k^3}{p}\right)\\
&\ll (\log x)^{\frac{1}{r_1}}.
\end{align*}

Now, for any $2\nmid d$ with $d\not\in \mathcal{D}$, there is some $f_d$ such that
\begin{align}\label{eq-lem-4-1}
\sum_{\substack{2^{ \lfloor m_1^{r_1}\rfloor}\equiv f_d\pmod{d}\\ m_1\le (\log 2x/\log 2)^{\frac{1}{r_1}}}}1 > \frac{(\log x)^{\frac{1}{r_1}}}{(\log\log x)^{k^3}}.
\end{align}
The congruence equation
$$
2^{ \lfloor m_1^{r_1}\rfloor}\equiv f_d\pmod{d}
$$
implies that there is some $1\le y_d< e_2(d)$ such that
\begin{align}\label{eq-lem-4-2}
\lfloor m_1^{r_1}\rfloor\equiv y_d\pmod{e_2(d)}.
\end{align}
By (\ref{eq-lem-4-1}), the number of $m_1$ satisfying the congruence is more than $\frac{(\log x)^{\frac{1}{r_1}}}{(\log\log x)^{k^3}}$. In this situation, we are going to show 
$$
e_2(d)\ll (\log x)^{1-\frac{1}{r_1}}(\log\log x)^{k^3}.
$$
In fact, let $M$ be the least gap between those $m_1$'s with $m_1\ge (\log x)^{\frac{1}{2r_1}}$ which satisfy congruence (\ref{eq-lem-4-2}). Then, clearly
$$
\frac{(\log x)^{\frac{1}{r_1}}}{(\log\log x)^{k^3}}\le (\log x)^{\frac{1}{2r_1}}+\frac{(\log 2x/\log 2)^{\frac{1}{r_1}}}{M}+1,
$$
which means
$$
M\le 2(\log\log x)^{k^3}.
$$
Moreover, by the definition of $M$ we have
$$
\lfloor m_1+M\rfloor^{r_1}-\lfloor m_1\rfloor^{r_1}\ge e_2(d)
$$
for any $m_1\ge (\log x)^{\frac{1}{2r_1}}$. Hence, 
$$
e_2(d)\le (m_1+M)^{r_1}-m_1^{r_1}+1\ll_{r_1} m_1^{r_1-1}M\ll (\log x)^{1-\frac{1}{r_1}}(\log\log x)^{k^3}.
$$

By the above analysis, we are left over to prove
\begin{align}\label{eq-lem-4-3}
\sum_{\substack{d=1,\ 2\nmid d \\P^+(d)\le \log x\\ e_2(d)\ll (\log x)^{1-r_1^{-1}}(\log\log x)^{k^3}}}^{\infty}\frac{\mu(d)^2 k^{3\omega(d)}}{d}\max_{f}\sum_{\substack{2^{ \lfloor m_1^{r_1}\rfloor}\equiv f\pmod{d}\\ m_1\le (\log 2x/\log 2)^{\frac{1}{r_1}}}}1\ll 
(\log x)^{\frac{1}{r_1}}.
\end{align}
We are in a position to use Lemma \ref{lem-3}. For any exponent pair $(\kappa, \lambda)$, we have
\begin{align*}
\sum_{\substack{d=1,\ 2\nmid d \\P^+(d)\le \log x\\ e_2(d)\ll (\log x)^{1-r_1^{-1}}(\log\log x)^{k^3}}}^{\infty}\frac{\mu(d)^2 k^{3\omega(d)}}{d}\max_{f}\sum_{\substack{2^{ \lfloor m_1^{r_1}\rfloor}\equiv f\pmod{d}\\ m_1\le (\log 2x/\log 2)^{\frac{1}{r_1}}}}1\ll S_1+S_2+S_3,
\end{align*}
where
$$
S_1=(\log x)^{\frac{1}{r_1}}\sum_{\substack{d=1,\ 2\nmid d \\P^+(d)\le \log x}}^{\infty}\frac{\mu(d)^2 k^{3\omega(d)}}{de_2(d)},
$$
$$
S_2=(\log x)^{\frac{r_1\kappa+\lambda}{r_1(1+\kappa)}}\sum_{\substack{d=1,\ 2\nmid d \\P^+(d)\le \log x}}^{\infty}\frac{\mu(d)^2 k^{3\omega(d)}}{de_2(d)^{\frac{\kappa}{1+\kappa}}},
$$
and
$$
S_3=\sum_{\substack{d=1,\ 2\nmid d \\P^+(d)\le \log x\\ e_2(d)\ll (\log x)^{1-r_1^{-1}}(\log\log x)^{k^3}}}^{\infty}\frac{\mu(d)^2 k^{3\omega(d)}e_2(d)^{\frac{1}{r_1}}}{d}.
$$
By Lemma \ref{lem-ET} with $B=k^3$ and $\varepsilon=1$, clearly we have
$
S_1\ll (\log x)^{\frac{1}{r_1}}.
$
It is also easy to bound $S_3$ as follows:
\begin{align*}
S_3&\ll \left((\log x)^{1-r_1^{-1}}(\log\log x)^{k^3}\right)^{\frac{1}{r_1}}\sum_{\substack{d=1,\ 2\nmid d \\P^+(d)\le \log x}}^{\infty}\frac{\mu(d)^2 k^{3\omega(d)}}{d}\\
&\le (\log x)^{\frac{1}{r_1}-\frac{1}{r_1^2}}(\log\log x)^{\frac{k^3}{r_1}}\prod_{p\le \log x}\left(1+\frac{k^3}{p}\right)\\
&\ll (\log x)^{\frac{1}{r_1}-\frac{1}{r_1^2}}(\log\log x)^{\frac{k^3}{r_1}+k^3}\\
&\ll (\log x)^{\frac{1}{r_1}}.
\end{align*}
For the estimate of $S_2$, using Lemma \ref{lem-ET} again we get
$$
S_2\ll (\log x)^{\frac{r_1\kappa+\lambda}{r_1(1+\kappa)}}.
$$
To complete the proof of (\ref{eq-lem-4-3}), it remains to pick an exponent pair $(\kappa, \lambda)$ such that
$$
\frac{r_1\kappa+\lambda}{1+\kappa}<1.
$$
For any $q\ge 2$ and $Q=2^q$, we know from \cite[p. 60]{Graham} that
$$
\left(\frac{1}{4Q-2},1-\frac{q+1}{4Q-2}\right)
$$
is an exponent pair. We now choose $q=\lfloor r_1\rfloor+1$. Then
$$
\frac{r_1\kappa+\lambda}{1+\kappa}<r_1\kappa+\lambda=\frac{r_1}{4Q-2}+1-\frac{q+1}{4Q-2}=1-\frac{2-(r_1-\lfloor r_1\rfloor)}{4Q-2}<1,
$$
completing the proof of (\ref{eq-lem-4-3}), and hence the whole lemma.
\end{proof}

The following technical lemma is quoted from Chen \cite[Theorem 3.1]{Chen-Xu}.

\begin{lemma}\label{lem-Chen}
Let $N$ and $r$ be positive integers, $a$ an integer, $K$ a positive number and $\varepsilon>0$. Then

{\rm (i)} the number of solutions of $y^{r}\equiv a\pmod{N}~(1\le y\le K)$ is less than 
$$
c\left(KN^{-\frac{1}{r}}+N^\varepsilon\right),
$$
where $c$ is a positive constant depending only on $r$ and $\varepsilon$;

{\rm (ii)} if $N\ge K^r$, the equation $y^{r}\equiv a\pmod{N}~(1\le y\le K)$ has at most one solution.
\end{lemma}

\begin{lemma}\label{last-lem}
Let $k>0 $ be a given number. Suppose that $r_1$ is a positive integer. Then we have
$$
\sum_{\substack{d=1,\ 2\nmid d \\P^+(d)\le \log x}}^{\infty}\frac{\mu(d)^2 k^{3\omega(d)}}{d}\max_{f}\sum_{\substack{2^{ m_1^{r_1}}\equiv f\pmod{d}\\ m_1\le (\log 2x/\log 2)^{\frac{1}{r_1}}}}1\ll 
(\log x)^{\frac{1}{r_1}},
$$
provided that $x$ is sufficiently large, where the implied constant depends at most on $r_1$.
\end{lemma}
\begin{proof}
For any given $f$, if there is at least one $m_i$ satisfy the congruence
$$
2^{  m_1^{r_1}}\equiv f\pmod{d},
$$
then there exists some $1\le a< e_2(d)$ such that
\begin{align*}
m_1^{r_1}\equiv a\pmod{e_2(d)}
\end{align*}
for all the solutions $m_1\le (\log 2x/\log 2)^{\frac{1}{r_1}}$. By Lemma \ref{lem-Chen} we have
\begin{align*}
\max_{f}\sum_{\substack{2^{ m_1^{r_1}}\equiv f\pmod{d}\\ m_1\le (\log 2x/\log 2)^{\frac{1}{r_1}}}}1\le c\left(\left(\frac{\log x}{e_2(d)}\right)^{\frac{1}{r_1}}+(\log x)^\varepsilon\right),
\end{align*}
where $\varepsilon>0$ is arbitrarily small and $c>0$ is a constant depending on $r_1$ and $\varepsilon$. Hence,
\begin{align}\label{eq-Chen-Xu-1}
\sum_{\substack{d=1,\ 2\nmid d \\P^+(d)\le \log x}}^{\infty}&\frac{\mu(d)^2 k^{3\omega(d)}}{d}\max_{f}\sum_{\substack{2^{ m_1^{r_1}}\equiv f\pmod{d}\\ m_1\le (\log 2x/\log 2)^{\frac{1}{r_1}}}}1\nonumber\\
&\ll (\log x)^{\frac{1}{r_1}}\sum_{\substack{d=1,\ 2\nmid d \\P^+(d)\le \log x}}^{\infty}\frac{\mu(d)^2 k^{3\omega(d)}}{d\big(e_2(d)\big)^{\frac{1}{r_1}}}+(\log x)^\varepsilon\sum_{\substack{d=1,\ 2\nmid d \\P^+(d)\le \log x}}^{\infty}\frac{\mu(d)^2 k^{3\omega(d)}}{d}.
\end{align}
By Lemma \ref{lem-ET} with $B=k^3$ and $\varepsilon=\frac{1}{r_1}$, the first term in the right-hand side of (\ref{eq-Chen-Xu-1}) is clearly $\ll (\log x)^{\frac{1}{r_1}}$. Using again the same argument as in Lemma \ref{lem-4}, we have
$$
\sum_{\substack{d=1,\ 2\nmid d \\P^+(d)\le \log x}}^{\infty}\frac{\mu(d)^2 k^{3\omega(d)}}{d}\le \prod_{p\le \log x}\left(1+\frac{k^3}{p}\right)\ll (\log\log x)^{k^3},
$$
giving the desired estimate of the second term in the right-hand side of (\ref{eq-Chen-Xu-1}).
\end{proof}

Now, we are ready to proceed with the proof of Theorem \ref{propo:2}.
\begin{proof}[Proof of Theorem \ref{propo:2}]
We prove the theorem by induction on $k$. First, for $k=1$ by the prime number theorem we clearly have
\begin{align*}
\sum_{n\le x}f_{\mathcal{A}}(n)=\sum_{\substack{p+a\le x\\ p\in \mathbb{P},\ a\in \mathcal{A}}}1
\le \Big(\sum_{p\le x}1\Big) \Big(\sum_{a\le x,\ a\in \mathcal{A}}1\Big)\ll \frac{x\mathcal{A}(x)}{\log x}.
\end{align*}
By the definition (\ref{eq-1-1}) of $\mathcal{A}$, we have
\begin{align}\label{eq-add-1}
\mathcal{A}(x)=\sum_{\substack{2^{ \lfloor m_1^{r_1}\rfloor}+\cdots+2^{\lfloor m_t^{r_t}\rfloor}\le x\\m_i\in \mathbb{N},\ 1\le i\le t}}1\le \prod_{1\le i\le t}\sum_{2^{\lfloor m_i^{r_i}\rfloor}\le x}1\ll (\log x)^{\frac{1}{r_1}+\cdots+\frac{1}{r_t}}.
\end{align}
Hence, we get
$$
\sum_{n\le x}f_{\mathcal{A}}(n)\ll x(\log x)^{\frac{1}{r_1}+\cdots+\frac{1}{r_t}-1},
$$
proving the case $k=1$.

Now, suppose that our theorem is true for all the integers $<k$. It suffices to prove the case $k$. In what follows, $p$'s will always be primes and $a$'s  are the elements of $\mathcal{A}$. Expanding the $k$-th moment of the representation function we will get
\begin{align}\label{eq-pro-1}
\sum_{n\le x}f_{\mathcal{A}}(n)^k=\sum_{p_1+a_1=\cdots=p_k+a_k\le x}1.
\end{align}
The primes above in the summation (\ref{eq-pro-1}) can be separated into a few groups as follows:
\begin{align}\label{eq-pro-2}
p_{11}=p_{12}=\cdots=p_{1s_1},\ p_{21}=p_{22}=\cdots=p_{2s_2},\ \cdots,\ p_{\ell1}=p_{\ell2}=\cdots=p_{\ell s_\ell},
\end{align}
where $\ell\le k$ and $s_1+s_2+\cdots+s_\ell=k$. For any $\ell<k$, the contribution of situation (\ref{eq-pro-2}) in the sum (\ref{eq-pro-1}) could be bounded as 
\begin{align}\label{eq-pro-3}
\ell !\sum_{a_{11}<a_{21}<\cdots<a_{\ell 1}\le x}~\sum_{\substack{p_{\ell1}<\cdots<p_{21}<p_{\ell 1}\\a_{11}+p_{11}=\cdots=a_{\ell 1}+p_{\ell 1}\le x}}1&\le \ell !\sum_{n\le x}f_{\mathcal{A}}(n)^\ell\nonumber\\
&\ll \ell! \cdot x (\log x)^{\ell \left(\frac{1}{r_1}+\cdots+\frac{1}{r_t}-1\right)},
\end{align}
where the last inequality comes from the inductive hypothesis. In view of (\ref{eq-pro-1}) and (\ref{eq-pro-3}) we obtain
\begin{align}\label{eq-pro-4}
\sum_{n\le x}f_{\mathcal{A}}(n)^k&\le k! \sum_{a_{1}<a_{2}<\cdots<a_{k}\le x}~\sum_{\substack{p_{k}<\cdots<p_{2}<p_{1}\\a_{1}+p_{1}=\cdots=a_{k}+p_{k}\le x}}1+\sum_{\ell=1}^{k-1}\ell! \cdot x (\log x)^{\ell \left(\frac{1}{r_1}+\cdots+\frac{1}{r_t}-1\right)}\nonumber\\
&\le k! \sum_{a_{1}<a_{2}<\cdots<a_{k}\le x}~\sum_{\substack{p_{k}<\cdots<p_{2}<p_{1}\\a_{1}+p_{1}=\cdots=a_{k}+p_{k}\le x}}1+k!\cdot x (\log x)^{k\left(\frac{1}{r_1}+\cdots+\frac{1}{r_t}-1\right)},
\end{align}
where the last inequality is guaranteed by the condition $\frac{1}{r_1}+\cdots+\frac{1}{r_t}\ge 1$.
Now, we see from (\ref{eq-pro-4}) that to complete the proof of our theorem, it would suffice to show
\begin{align}\label{eq-pro-5}
\sum_{a_{1}<a_{2}<\cdots<a_{k}\le x}~\sum_{\substack{p_{k}<\cdots<p_{2}<p_{1}\\a_{1}+p_{1}=\cdots=a_{k}+p_{k}\le x}}1\ll x (\log x)^{k\left(\frac{1}{r_1}+\cdots+\frac{1}{r_t}-1\right)}.
\end{align}

It is clear that
\begin{align*}
\sum_{\substack{p_{k}<\cdots<p_{2}<p_{1}\\a_{1}+p_{1}=\cdots=a_{k}+p_{k}\le x}}1\le \sum_{\substack{p_k\le x\\ p_i=p_k+a_k-a_i,\ 1\le i\le k-1}}1.
\end{align*}
By Lemma \ref{lem-2} with $g=k-1$ and $q_i=1$ we get
\begin{align*}
\sum_{\substack{p_{k}<\cdots<p_{2}<p_{1}\\a_{1}+p_{1}=\cdots=a_{k}+p_{k}\le x}}1\ll \prod_{p|E}\left(1-\frac{1}{p}\right)^{\rho(p)-k+1}\prod_{p|b_1\cdots b_{k-1}}\left(1-\frac{1}{p}\right)^{-1}\frac{x}{(\log x)^{k}},
\end{align*}
where $b_i=a_k-a_i$ and 
$$
E=\prod_{1\le h<s\le k-1}\big(b_s-b_h\big)=\prod_{1\le h<s\le k-1}\big(a_h-a_s\big).
$$
Therefore, we conclude that
\begin{align}\label{eq-pro-6}
\sum_{a_{1}<a_{2}<\cdots<a_{k}\le x}~\sum_{\substack{p_{k}<\cdots<p_{2}<p_{1}\\a_{1}+p_{1}=\cdots=a_{k}+p_{k}\le x}}1&\ll \frac{x}{(\log x)^{k}}\sum_{a_{1}<a_{2}<\cdots<a_{k}\le x}\prod_{p|E_k}\left(1-\frac{1}{p}\right)^{-k}\nonumber\\
&\ll \frac{x}{(\log x)^{k}}\sum_{a_{1}<a_{2}<\cdots<a_{k}\le x}\prod_{p|E_k}\left(1+\frac{k}{p}\right),
\end{align}
where 
$$
E_k=Eb_1b_2\cdots b_{k-1}=\prod_{1\le h<s\le k}\big(a_s-a_h\big).
$$

Trivially, the right-hand side of (\ref{eq-pro-6}) is $\ll x(\log x)^{k\left(\frac{1}{r_1}+\cdots+\frac{1}{r_t}-1\right)}(\log\log x)^k$ since 
$$
\prod_{p|E_k}\left(1+\frac{k}{p}\right)\ll (\log\log x)^k.
$$
Thus, we have to earn a saving of $(\log\log x)^k$ factor in the estimate to fulfill our wish. Following Erd\H os'  idea \cite[Theorem 2]{Erdos2}, we employ the inequality of the geometric and power means 
$$
x_1x_2\cdots x_m\le \frac{x_1^m+x_2^m+\cdots+x_m^m}{m}
$$
to our case, and then we will get
\begin{align}\label{eq-pro-7}
\prod_{p|E_k}\left(1+\frac{k}{p}\right)\le \frac{1}{\binom{k}{2}}\sum_{1\le h<s\le k}\prod_{p|(a_s-a_h)}\left(1+\frac{k}{p}\right)^{\binom{k}{2}}\ll \sum_{1\le h<s\le k}\prod_{p|(a_s-a_h)}\left(1+\frac{k^3}{p}\right),
\end{align}
where the implied constant depends on $k$. Taking (\ref{eq-pro-7}) into (\ref{eq-pro-6}) we obtain
\begin{align*}
\sum_{a_{1}<a_{2}<\cdots<a_{k}\le x}~\sum_{\substack{p_{k}<\cdots<p_{2}<p_{1}\\a_{1}+p_{1}=\cdots=a_{k}+p_{k}\le x}}1&\ll \frac{x}{(\log x)^{k}}\sum_{a_{1}<a_{2}<\cdots<a_{k}\le x}~\sum_{1\le h<s\le k}\prod_{p|(a_s-a_h)}\left(1+\frac{k^3}{p}\right)
\nonumber\\
&\le \frac{x}{(\log x)^{k}}\binom{k}{2}\mathcal{A}(x)^{k-2}\sum_{a_1<a_2}\prod_{p|(a_1-a_2)}\left(1+\frac{k^3}{p}\right).
\end{align*}
Now, using the bound (\ref{eq-add-1}) of $\mathcal{A}(x)$ we have
\begin{align*}
\sum_{a_{1}<a_{2}<\cdots<a_{k}\le x}~\sum_{\substack{p_{k}<\cdots<p_{2}<p_{1}\\a_{1}+p_{1}=\cdots=a_{k}+p_{k}\le x}}1\ll x(\log x)^{(k-2)\left(\frac{1}{r_1}+\cdots+\frac{1}{r_t}\right)-k}\sum_{a_1<a_2\le x}\prod_{p|(a_1-a_2)}\left(1+\frac{k^3}{p}\right).
\end{align*}
Hence, to complete the proof of (\ref{eq-pro-5}) it remains to show
\begin{align*}
\sum_{a_1<a_2\le x}\prod_{p|(a_1-a_2)}\left(1+\frac{k^3}{p}\right)\ll (\log x)^{2\left(\frac{1}{r_1}+\cdots+\frac{1}{r_t}\right)}.
\end{align*}
We now make use of a trick of Elsholtz, Luca and Planitzer \cite{Elsholtz-Luca-Plantizer}. We separate the primes $p$ in the product of (\ref{eq-pro-8}) according to $p>\log x$ or $p\le \log x$. For $a_1,a_2\le x$, it is not hard to see
$$
\prod_{\substack{p|(a_1-a_2)\\ p>\log x}}\left(1+\frac{k^3}{p}\right)\le \left(1+\frac{k^3}{\log x}\right)^{\frac{\log x}{\log 2}}\ll 1.
$$
Thus, we only need to prove
\begin{align}\label{eq-pro-8}
\sum_{a_1<a_2\le x}\prod_{\substack{p|(a_1-a_2)\\2<p\le \log x}}\left(1+\frac{k^3}{p}\right)\ll (\log x)^{2\left(\frac{1}{r_1}+\cdots+\frac{1}{r_t}\right)}.
\end{align}

Expanding the product and then exchanging the order of summations, we have
\begin{align}\label{eq-pro-9}
\sum_{a_1<a_2\le x}\prod_{\substack{p|(a_1-a_2)\\2<p\le \log x}}\left(1+\frac{k^3}{p}\right)\le 
\sum_{\substack{d=1 \\P^+(d)\le \log x\\ 2\nmid d}}^{\infty}\frac{\mu(d)^2 k^{3\omega(d)}}{d}\sum_{\substack{a_1<a_2\le x\\ a_1\equiv a_2\pmod{d}}}1,
\end{align}
where $\mu(d)$ is the M\"obius function, $\omega(d)$ is the number of different prime factors of $d$, and $P^+(d)$ is the largest prime factor of $d$. By the definition of $\mathcal{A}$, we know
\begin{align*}
\sum_{\substack{a_1<a_2\le x\\ a_1\equiv a_2\pmod{d}}}1&=\sum_{\substack{2^{ \lfloor m_1^{r_1}\rfloor}+\cdots+2^{\lfloor m_t^{r_t}\rfloor}\equiv 2^{ \lfloor n_1^{r_1}\rfloor}+\cdots+2^{\lfloor n_t^{r_t}\rfloor}\pmod{d}\\2^{ \lfloor m_1^{r_1}\rfloor}+\cdots+2^{\lfloor m_t^{r_t}\rfloor}<2^{ \lfloor n_1^{r_1}\rfloor}+\cdots+2^{\lfloor n_t^{r_t}\rfloor}\le x\\m_i, n_i\in \mathbb{N},\ 1\le i\le t}}1\nonumber\\
&\ll (\log x)^{2\left(\frac{1}{r_1}+\cdots+\frac{1}{r_t}\right)-\frac{1}{r_1}}\max_{f}\sum_{\substack{2^{ \lfloor m_1^{r_1}\rfloor}\equiv f\pmod{d}\\ m_1\le (\log 2x/\log 2)^{\frac{1}{r_1}}}}1,
\end{align*}
which together with (\ref{eq-pro-8}) implies 
\begin{align}\label{eq-pro-10}
\sum_{a_1<a_2\le x}&\prod_{\substack{p|(a_1-a_2)\\2<p\le \log x}}\left(1+\frac{k^3}{p}\right)\nonumber\\
&\ll (\log x)^{2\left(\frac{1}{r_1}+\cdots+\frac{1}{r_t}\right)-\frac{1}{r_1}}\sum_{\substack{d=1,\ 2\nmid d \\P^+(d)\le \log x}}^{\infty}\frac{\mu(d)^2 k^{3\omega(d)}}{d}\max_{f}\sum_{\substack{2^{ \lfloor m_1^{r_1}\rfloor}\equiv f\pmod{d}\\ m_1\le (\log 2x/\log 2)^{\frac{1}{r_1}}}}1.
\end{align}
We will separate the remaining argument into two cases.

{\it Case I.} At least one of $r_i$ is not an integer. Changing the order of $r_i$'s if necessary, we may assume without loss of generality that $r_1$ is not an integer. Then, by (\ref{eq-pro-10}) and Lemma \ref{lem-4},
\begin{align*}
\sum_{a_1<a_2\le x}&\prod_{\substack{p|(a_1-a_2)\\2<p\le \log x}}\left(1+\frac{k^3}{p}\right)\ll (\log x)^{2\left(\frac{1}{r_1}+\cdots+\frac{1}{r_t}\right)-\frac{1}{r_1}}(\log x)^{\frac{1}{r_1}}\ll (\log x)^{2\left(\frac{1}{r_1}+\cdots+\frac{1}{r_t}\right)},
\end{align*}
establishing (\ref{eq-pro-8}) in {\it Case I.}

{\it Case II.} All of the $r_i$'s are integers. By (\ref{eq-pro-10}) and Lemma \ref{last-lem}, we also have
\begin{align*}
\sum_{a_1<a_2\le x}&\prod_{\substack{p|(a_1-a_2)\\2<p\le \log x}}\left(1+\frac{k^3}{p}\right)\ll (\log x)^{2\left(\frac{1}{r_1}+\cdots+\frac{1}{r_t}\right)},
\end{align*}
completing the proof of (\ref{eq-pro-8}) in {\it Case II,} and hence the whole theorem.
\end{proof}

As we mentioned in Remark \ref{remark1}, if $p$ is a prime and $n=p+2$, then 
$
n+2=p+2^2.
$
This leads us to the following problem which we cannot solve at present whose positive answer can be viewed as the first step toward Conjecture \ref{conjecture-1}.

\begin{problem}
Prove unconditionally that
$$
\lim_{x\rightarrow\infty}\frac{\#\left\{n\le x: \text{both~} n~\text{and~} n+2 ~\text{can~be~written~as~} p+2^m\right\}}{x/\log x}=\infty.
$$
\end{problem}

\section*{Acknowledgments}
We thank Alfr\'ed R\'enyi Institute of Mathematics for providing us an excellent academic environment. We also thank the support from CSC programs.

\end{document}